\documentclass[letterpaper, 10 pt, conference]{ieeeconf}  

\IEEEoverridecommandlockouts    

\usepackage{graphics} 
\usepackage{amsmath} 
\usepackage{amssymb}  
\usepackage{subfigure}
\usepackage{stfloats}
\usepackage{graphicx}
\usepackage{algorithm,algorithmic}

\newcommand{\X}{\mathbb{X}}
\newcommand{\R}{\mathbb{R}}
\newcommand{\N}{\mathbb{N}}
\newcommand{\C}{\mathbb{C}}

\newcommand{\K}{\mathcal{K}}
\newcommand{\HS}{\mathcal{H}}

\usepackage[dvipsnames]{xcolor}
\usepackage[space]{cite}

\newtheorem{defn}{Definition}
\newtheorem{prop}{Proposition}

\newtheorem{assmpt}{Assumption}
\newtheorem{ex}{Example}
\newtheorem{remark}{Remark}%

\title{\LARGE \bf RKHS method for computing Koopman-based Lyapunov functions
}

\author{François-Grégoire Bierwart$^{1}$ and Alexandre Mauroy$^{2}$
\thanks{$^{1,2}$ Department of Mathematics and Namur Research Institute for Complex Systems (naXys), University of Namur, 5000, Belgium
        {\tt\small francois-gregoire.bierwart.@unamur.be, alexandre.mauroy@unamur.be}}%
}

\begin{document}

\maketitle
\thispagestyle{empty}
\pagestyle{empty}


\begin{abstract}
The Koopman operator is a powerful approach to global stability analysis of nonlinear systems, which provides a systematic procedure for Lyapunov function design. In this framework, Lyapunov functions are obtained through the eigenfunctions of the Koopman operator associated with the eigenvalues of the Jacobian matrix at the equilibrium. In practice, the eigenfunctions are approximated via a finite-dimensional representation of the operator, and there is no guarantee that the approximated spectrum accurately matches the true one. In this paper, we develop a kernel-based method to compute Koopman eigenfunctions and preserve the spectrum of the Jacobian matrix. This approach is suitable for stability analysis of high-dimensional systems thanks to the kernel trick. Moreover, the Lyapunov function candidate is validated through a scenario-based optimization technique that provides a reliable estimation of the region of attraction of the system.  
\end{abstract}

\section{Introduction}

Global stability analysis of dynamical systems plays an important role in many areas of sciences. For instance, estimating the region of attraction is vital for safety-critical systems, such as flight control or medical applications. It is well known that the global stability property of a system is fully described by a Lyapunov function \cite{KhalilNonlinearControl}. For nonlinear systems, and in particular nonlinear ones, the design of a Lyapunov function is challenging and has attracted great attention over the last decades. For instance, several Lyapunov function design methods have been developed, such as sum-of-squares (SOS) methods \cite{papachristodoulou2005analysis}, solution to the Zubov's equation \cite{vannelli1985maximal,liu2025physics} and collocation kernel-based methods \cite{giesl2007meshless,giesl2016approximation}, to list a few. We refer to \cite{giesl2015review} for a general overview of existing methods for Lyapunov function design.\\
The Koopman operator plays an important role in global stability analysis, as it allows to design systematic Lyapunov functions through Koopman eigenfunctions associated with the eigenvalues of the Jacobian matrix at the equilibrium \cite{mauroy2014global}. Although this framework is appealing, it relies on numerical approximation of the eigenfunctions \cite[Chapter 1]{mauroy2020koopman}. In fact, the infinite-dimensional Koopman operator is approximated as a finite-dimensional operator by restricting the space of observables to a finite-dimensional subspace. Several approaches have been considered in that context, including polynomial expansion  \cite{mauroy2014global,yadav2025approximation} or kernel-based methods \cite{das2020koopman,klus2020kernel,lee2025kernel,williams2014kernel}, to list a few. These latter methods have received increasing attention in recent years and scale well for high-dimensional systems due to the kernel trick, in contrast to polynomial subspaces which suffer from the curse of dimensionality. As a related issue, there is usually no guarantee that the relevant subset of the point spectrum required to compute the Lyapunov function is preserved by the approximation. However, this \emph{spectral pollution} problem can be overcome by imposing the desired eigenvalues when solving the eigenvalue equation, see e.g. \cite{mauroy2014global} or more recently \cite{lee2025kernel}.

In this paper, we design a new Koopman-based method to estimate Koopman eigenfunctions and construct Lyapunov functions, which relies on Reproducing Kernel Hilbert Space (RKHS) theory. Thanks to a specific orthogonal projection inspired by \cite{lee2025kernel}, in a well-chosen RKHS, this method not only allows to preserve the spectrum of the Jacobian matrix at the equilibrium, but is also well-suited to high-dimensional systems. Moreover, the Lyapunov function candidate derived from the method is further validated through a scenario-based method \cite{campi2015non}. Numerical tests demonstrate that the method is efficient for estimating the region of attraction (ROA) of low and high-dimensional systems.\\
The paper is organized as follow. In Section \ref{sec:2}, we review the fundamentals of global stability analysis through the Koopman operator framework. The main results of the paper are given in Section \ref{sec:method}, where we develop a kernel-based method to approximate the eigenfunctions of the Koopman operator while preserving the spectrum of the Jacobian matrix. The Lyapunov candidate resulting from the approximation is further validate in Section \ref{sec:validation} through scenario-based techniques to estimate the region of attraction of an equilibrium. Section \ref{sec:examples} provides numerical experiments for different systems and conclusions and perspectives are given in Section \ref{sec:ccl}

\section{Koopman operator framework for global stability analysis}\label{sec:2}

\subsection{Koopman-based Lyapunov functions}

\noindent In this section, we review the basics of global stability analysis through the Koopman operator. Consider the system 
\begin{equation}\label{sysdyn:def}
\dot{x} = F(x), \quad x\in\X
\end{equation}
where $\X\subset\mathbb{R}^n$ is a compact set containing the origin and $F$ is Lipchitz and analytic in $\mathbb{X}$. We will refer to $\varphi^t : \mathbb{R}^+\times \X \rightarrow \mathbb{R}^n,~ x \mapsto \varphi^t(x) = x(t)$ as the flow map of \eqref{sysdyn:def} and we assume that the system admits a locally stable hyperbolic equilibrium point at the origin whose region of attraction is denoted by $\mathcal{R}$, i.e,  
$$
\mathcal{R} ~=~ \left\{x\in\X \mid \lim_{t \rightarrow \infty}\varphi^t(x) = 0\right\}.
$$
It is well known that $\mathcal{R}$ is fully characterized by a Lyapunov function for the system \cite{KhalilNonlinearControl}. In this paper, we will focus on Koopman-based Lyapunov functions as outlined in \cite{mauroy2014global}. We first recall the definition of the semigroup of Koopman operators.
\begin{defn}\label{def:koop}
Let $\mathcal{F}= C^1(\mathcal{R})$. We define the Koopman semigroup as the family $\{\K_t\}_{t\geq 0}$, $\K_t : \mathcal{F} \rightarrow \mathcal{F}$ such that
$
\K_t f =f \circ \varphi^{t}$ for all $f \in \mathcal{F}.
$
\end{defn}
For any differentiable $f$ and for all $x\in\mathcal{R}$, we can also define
$$
(\mathcal{L}f)(x) ~ \triangleq ~ \lim_{t \rightarrow 0^{+}}\dfrac{\mathcal{K}_tf(x) - f(x)}{t} ~=~ F(x)\cdot\nabla f(x).
$$
If the Koopman semigroup $\left(\K_t\right)_t$ is strongly continuous, the operator $\mathcal{L}$ is its infinitesimal generator\footnote{For the sake of simplicity, we will refer to $\mathcal{L}$ as the Koopman generator, even if $\{\K_t\}_t$ is not strongly continuous.}. Since the vector field is supposed to be known in our context, we will focus on the generator rather than on the semigroup.

The (linear) Koopman generator $\mathcal{L}$ and semigroup $\K_t$ share the same eigenfunctions $\phi_{\lambda}\in\mathcal{F}$ associated with the eigenvalues $\lambda$ and $e^{\lambda t}$, respectively, i.e. $\mathcal{L}\phi_{\lambda} = \lambda \phi_{\lambda}$ and $\K_t\phi_\lambda = e^{\lambda t}\phi_\lambda$. Since the origin is a hyperbolic equilibrium and $\mathcal{F}= C^1(\mathcal{R})$, the point spectrum of $\mathcal{L}$ contains the spectrum $\sigma(\mathbf{J}_F(0))$ of the Jacobian matrix $\mathbf{J}_F(0)$ at the origin (see e.g. \cite{mauroy2014global}). Moreover, the corresponding eigenfunctions are unique (see \cite{kvalheim2021existence}) and allow us to construct a generic Lyapunov function of the form
\begin{equation}\label{TrueLyap}
V(x) ~=~ \sum_{i=1}^n\left|\phi_{\lambda_i}(x)\right|^2,~~~\forall x\in \mathcal{R},
\end{equation}
with $\lambda_i\in\sigma(\mathbf{J}_F(0))$ and $\Re(\lambda_i)<0$ for all $i=1,\ldots,n$. 

\subsection{Numerical approximation of the Koopman generator}

\noindent The Koopman-based Lyapunov function \eqref{TrueLyap} relies on the eigenfunctions of $\mathcal{L}$, which are typically unknown. This issue can be addressed by computing an approximation of $\mathcal{L}$. Let $\mathcal{F}_\ell = \mathrm{span}\{\psi_i\}_{i=1}^\ell\subset \mathcal{F}$ be an $\ell$-dimensional subspace, for some $\ell\in\N_0$, where $\{\psi_i\}_{i=1}^\ell$ is a set of basis functions. For a given projection operator $\Pi : \mathcal{F} \rightarrow \mathcal{F}_\ell$, an approximation of $\mathcal{L}$ is given by $\mathcal{L}_\ell = \Pi\mathcal{L}\hspace{-0.1cm}\mid _{\mathcal{F}_\ell} : \mathcal{F}_\ell \rightarrow \mathcal{F}_\ell$. Since $\mathcal{L}_\ell$ is finite-dimensional, it admits the matrix representation $\bold{L}_\ell$ whose $j^{\text{th}}$ column contains the components of $\Pi(\mathcal{L}\psi_j)$ in the basis $\{\psi_i\}_{i=1}^\ell$ \cite[Chapter 1]{mauroy2020koopman}. Let $\{\widetilde{\lambda_i}\}_{i=1}^\ell$ be the eigenvalues of $\mathcal{L}_\ell$ and let $\{\widetilde{\phi}_{\,\widetilde{\lambda}_i}\}_{i=1}^\ell$ be the associated eigenfunctions. It can be shown that, for any $i=1,\ldots,\ell$, $\widetilde{\phi}_{\widetilde{\lambda_i}}(x) = \Psi(x) \, v_i$ where $\Psi(x) = [\psi_1(x),\ldots,\psi_\ell(x)]$ and $v_i$ is the right eigenvector of $\bold{L}_\ell$ associated with the eigenvalue $\widetilde{\lambda}_i$. Then, provided that there exist eigenvalues $\widetilde{\lambda}_i \approx \lambda_i$, the eigenfunctions $\widetilde{\phi}_{\,\widetilde{\lambda}_i}$ can be used to approximate the exact eigenfunctions, yielding in turn an approximation of the Lyapunov function \eqref{TrueLyap}. However, there is no guarantee that the approximated eigenvalues (and thus the eigenfunctions) approximate well those of $\mathcal{L}$, except in very few cases. For instance, in our previous work \cite{bierwarterror}, the analyticity of the eigenfunctions was leveraged to ensure that $\sigma(\mathbf{J}_F(0))\subset\sigma(\mathcal{L})$, and approximation error bounds were obtained. Yet, this framework is limited since it involves the use of monomials, which is only suitable to low-dimensional systems. Moreover, error bounds are only valid in the domain of analyticity which could be conservative with respect to $\mathcal{R}$.  

\section{RKHS-based spectral approximation of the Koopman generator}\label{sec:method}

In this section, we construct an approximation of the Koopman generator which is well-suited to high-dimensional systems while ensuring a good approximation of the eigenvalues of $\mathcal{L}$. To this end, we will rely on the theory of Reproducing Kernel Hilbert Spaces (RKHS).

\subsection{A brief introduction to RKHS}

\noindent We first provide a brief overview of Reproducing Kernel Hilbert Space (RKHS) theory. We refer to \cite{paulsen2016introduction} for a more thorough analysis on RKHS.
\begin{defn}
A Reproducing Kernel Hilbert Space $\mathcal{H}$ on $\X$ is a Hilbert space of functions $f : \X \rightarrow \C$ where the evaluation functional $E_x : \mathcal{H} \rightarrow \C, f \mapsto E_x(f) = f(x)$ is bounded for any $x\in\X$.
\end{defn}
\noindent It follows from the Riesz representation theorem that, for any \mbox{$x\in\X$}, there exists $k_x\in \mathcal{H}$ such that \mbox{$\langle f,k_x\rangle_{\HS} = f(x)$} for all $f\in \mathcal{H}$. The function $k : \X\times\X \rightarrow \C$, \mbox{$(x,y) \mapsto k(x,y) := k_x(y)$} is the unique \textit{reproducing kernel} of $\mathcal{H}$. That is, a RKHS is uniquely defined through its kernel function. As an example, the polynomial kernel 
$$k(x,y) ~=~ (x^\top y + c)^d ~=~ \sum_{s = 0}^d{d\,\choose s} \,c^{d-s}\,(x^\top y)^s$$
with $c\geq 0$ and $d\in\mathbb{N}$ is associated with the RKHS of polynomials up to order $d$. The following result extends the above example to a whole class of analytic kernels (called Taylor-based kernels).
\begin{prop}[\hspace{-0.0001cm}\cite{shawe2004kernel}]\label{prop-kernel}
For any analytic function $f$ with nonnegative Taylor coefficients, $f(x^\top y)$ is a valid kernel.
\end{prop}

\noindent We also recall a useful result on interpolation in RKHS.

\begin{prop}[\hspace{-0.0001cm}\cite{paulsen2016introduction}]\label{interp}
Let $\mathcal{H}$ be a RKHS on $\X$ with reproducing kernel $k$, and let $Q = \{x_1,\ldots,x_m\}\subset \X$ be a set of distinct points. If the matrix $[\mathbf{G}]_{ij}:=[k(x_i,x_j)]_{ij}$ is invertible, then for any $v = (v_1,\ldots,v_m)^\top\subset \C^m$ there exists a function $g \in \mathrm{span} \{k_{x_1},\dots,k_{x_m}\}$ interpolating these values, which is given by 
$
g = \sum_j \alpha_jk_{x_j},
$
where $\alpha := (\alpha_1,\ldots,\alpha_m)^\top$ is the solution of the linear system $\mathbf{G}\alpha = v$.  
\end{prop}

Let $Q := \{q_i\}_{i=1}^m \subset \X$ ($m\in\N_0)$ and $\mathcal{H}$ be a RKHS on $\X$ with the kernel $k$. Then, for any function $g$, we denote $\mathcal{I}_{Q,k}(g)$ as the interpolant in Proposition \ref{interp} with $v = \{g(q_i)\}_i$. According to Proposition \ref{interp}, a closed-form expression for the interpolation can be obtained provided that the Gram matrix $\mathbf{G}$ is invertible. The following result shows that some Taylor-based kernels satisfy this condition.
\begin{prop}[\hspace{-0.0001cm}\cite{pinkus2004strictly}]\label{propo:invert_ker}
Let $f(x) = \sum_{i=0}^\infty a_ix^i, x\in\mathbb{R}$ with $a_i\geq0$ for all $i$. Let $\X\subset\R^n$ and assume that $\X$ contains the set $\{\alpha x \mid |\alpha| \leq c\}$ for some $c>0$ and for some $x\in \X$ such that $\|x\| = 1$. Let $k(x,y) = f(x^\top y)$. Then, for any set of distinct collocation points $Q := \{q_i\}_{i=1}^m \subset \X$, the Gram matrix $[\mathbf{G}]_{ij} = [k(q_i,q_j)]_{ij}$ is strictly positive definite if and only if $a_0>0$ and there is an infinite number of even integers $i$ and an infinite number of odd integers $i$ such that $a_i>0$.
\end{prop}


\subsection{Spectral approximation of the Koopman generator}

In this section, we design a RKHS-based method to approximate the Koopman generator and its eigenfunctions while preserving $\sigma(\mathbf{J}_F(0))$. To do so, we will consider the RKHS $\HS_1 = \mathrm{span}\{x^\alpha\}_{|\alpha|=1}$ (associated to the RKHS with the kernel $k_1(x,y) = x^\top y$)\footnote{Unless otherwise stated, $x^\alpha$, $\alpha\in\mathbb{N}^n$, refers to a monomial $x^\alpha = x_1^{\alpha_1}\ldots\,x_n^{\alpha_n}$ with \mbox{$|\alpha| = \alpha_1+\cdots+\alpha_n$}.} and $\HS_2$ (associated with a kernel $k_2$) on $\mathcal{R}\subset\X$, such that $\HS_1\oplus\HS_2\subseteq\mathcal{F}$. In addition, we make the following assumption.

\begin{assmpt}\label{assmpt:1}
Let $Q := \{q_i\}_{i=1}^m \subset \mathbb{X}$ be a set of distinct collocation points. We assume that (i) $Q\subset \mathcal{R}$ and (ii) the Gram matrix $[\mathbf{G}]_{ij} = k_2(q_i,q_j)$ is invertible.\end{assmpt}
\begin{remark}\label{rem:samplingQ}
According to Assumption \ref{assmpt:1}, the set of collocation points $Q$ must lie in the (unknown) region of attraction. In practice, the sampled points are taken from the largest hypercube $[-a,a]^n$ providing the best results for the approximation of the ROA. These results can be further improved by selecting points in a larger set and discarding those that do not belong to the region of attraction according to numerical simulations. Such a trick does not fit with our original problem setting, but is illustrated in Section \ref{sec:examples} for informative purposes.
\end{remark}
Note that $\mathcal{R}$ contains a ball centered at the origin and therefore verifies the assumptions of Proposition \ref{propo:invert_ker}. Provided that $k_2$ is analytic on $\mathcal{R}$, we only need to verify the assumptions on the Taylor coefficients of $k_2$ to ensure that the Gram matrices are invertible. The following example provide a kernel for $k_2$ that verify such assumptions.
\begin{ex}\label{example1}
Let $\HS = \mathrm{span}\{x^\alpha\}_{|\alpha|\geq0}$ be the Fock space of analytic entire functions on $\mathcal{R}$ with the kernel 
$$k(x,y) = e^{\eta(x^\top y)} 
$$ 
for some $\eta > 0$. Using Taylor series expansion, we can split the kernel as  
$$
k(x,y) = \eta(x^\top y) + \hspace{-0.2cm}\sum_{k\in\mathbb{N}\setminus\{1\}}\hspace{-0.2cm}\dfrac{\eta^k(x^\top y)^k}{k!} := \eta k_1(x,y) + k_2(x,y).
$$
It is clear that $\HS_1\cap\HS_2 = \{0\}$. Moreover, $[\mathbf{G}]_{ij} = k_2(q_i,q_j)$ is invertible for any $Q=\{q_i\}_i\subset\mathcal{R}$ since $k_2$ verifies the assumptions of Proposition \ref{propo:invert_ker}.
\end{ex}

\noindent Let $\HS_2$ satisfy Assumption \ref{assmpt:1} and $S\subset \mathcal{F}$ be the subspace
\begin{equation}\label{eq:directsum}
S ~:=~ \mathcal{H}_1\oplus~\mathcal{H}_{2,m},\vspace{0.2cm}
\end{equation}
where $\mathcal{H}_{2,m} = \mathrm{span}\{{k_2(\cdot,q_i)}\}_{i=1}^m$ and $Q = \{q_i\}_{i=1}^m \subseteq \X$ is a finite set of distinct points satisfying Assumption \ref{assmpt:1}. We construct the approximation $\mathcal{L}_\ell ~:=~ \Pi\,\mathcal{L}\hspace{-0.12cm}\mid_{S} ~: S \rightarrow S$ where the projection $\Pi : \mathcal{F} \rightarrow S$ is inspired by \cite{lee2025kernel} and given \footnote{We can easily verify that $\Pi$ is a projection since $\Pi\circ\Pi = \Pi$.} as 
\begin{equation}\label{new_proj}
\Pi g ~=~ \nabla^\top  g\mid_0x~+~ \mathcal{I}_{Q,k_2}(g-\nabla^\top g\mid_0 x)
\end{equation}

\noindent for any differentiable $g$. Under some additional assumptions, we can show that the projection operator in \eqref{new_proj} is an orthogonal projection if restricted to $\HS_1\oplus\HS_2$. 

\begin{prop}\label{prop:proj_ortho}
Suppose that $\HS_1 \oplus \HS_2 \subseteq \mathcal{F}$  and $\HS_2$ satisfies Assumption \ref{assmpt:1}, with $k_2(x,y) = \sum_{k\in\mathbb{N}\setminus\{1\}}a_k(x^\top y)^k$, $a_k > 0$, for any $(x,y)\in \mathcal{R}\times\mathcal{R}$. Then, $\Pi|_{\HS_1 \oplus \HS_2}$ is an orthogonal projection in $\HS_1 \oplus \HS_2$.
\end{prop}


\begin{proof}
Consider $g\in\HS_1\oplus \HS_2$. Then, according to \cite[Corollary 5.8]{paulsen2016introduction}, $\HS_1\oplus \HS_2$ is a RKHS with reproducing kernel $k = k_1+k_2$, which is equipped with an inner product such that $\HS_1\perp \HS_2$. Moreover, the structure of $k_2$ implies that $\HS_2 = \mathrm{span}\{x^\alpha\}_{|\alpha|\neq 1}$ so that any function in $\HS_1\oplus\HS_2$ is analytic. Hence, we have 
$P_{\HS_1} g = \nabla^\top g \mid_0x \in\HS_1$ and $P_{\HS_2} g = g-\nabla^\top g \mid_0x \in \HS_2$, where $P_{\HS_1}$ and $ P_{\HS_2}$ are the orthogonal projections onto $\HS_1$ and $\HS_{2}$, respectively. Finally, it follows that
$$\Pi g = P_{\HS_1}g + \mathcal{I}_{Q,k_2} P_{\HS_{2}} g = (P_{\HS_1}+ P_{\HS_{2,m}} P_{\HS_{2}}) g$$ 
since $\mathcal{I}_{Q,k_2}|_{\HS_2}=P_{\HS_{2,m}}$, where $P_{\HS_{2,m}}:\HS_2 \to \HS_{2,m}$ is an orthogonal projection (see e.g. \cite[Proposition 4.2]{paulsen2016introduction}).
\end{proof}


\noindent An example of RKHS that satisfies the assumptions of Proposition \ref{prop:proj_ortho} is the Fock space (see Example \ref{example1}).
\noindent The set of assumptions in Proposition \ref{prop:proj_ortho} is the key to the design of our method.  Assume that such assumptions are satisfied. Since $F$ is an analytic vector field admitting a hyperbolic origin, we can easily show that the matrix representation of $\mathcal{L}_\ell$ has the following structure
$$
\mathbf{L}_\ell =
\left[
\begin{array}{cc}
\mathbf{J}_F(0)^\top & 0\\[0.2cm]
\mathbf{G}^{-1}\mathbf{N} & \mathbf{G}^{-1}\mathbf{A}\\
\end{array}
\right],
$$
\noindent where $\mathbf{N}_{ij} = N_j(q_i)$, $\mathbf{A}_{ij} = (\mathcal{L}k_2(q_i,\cdot))(q_j)$ and \mbox{$N(x) = F(x)-\mathbf{J}_F(0)x$} is the nonlinear part of $F$. This structure is block-triangular so that the approximation allows to recover exactly the eigenvalues of $\mathbf{J}_F(0)$. Moreover, the eigenvectors can be computed recursively. Indeed, let $v_\lambda = [v_{\lambda,1},v_{\lambda,2}]^\top$ be the right eigenvector of $\mathbf{L}_{\ell}$ associated with the eigenvalue $\lambda$. We have, $v_{\lambda,1} = w_\lambda$ where $w_\lambda$ is the left eigenvector of $\mathbf{J}_F(0)$ and $v_{\lambda,2}$ is the solution to the linear system $(\mathbf{G}^{-1}\mathbf{A}-\lambda I)v_{\lambda,2} = -\mathbf{G}^{-1}\bold{N}w_\lambda$ or equivalently $(\mathbf{A}-\lambda \mathbf{G})v_{\lambda,2} = -\bold{N}w_\lambda$. If $(\mathbf{A}-\lambda \mathbf{G})$ is singular, we can still compute the least squares solution $v_{\lambda,2} = -(\mathbf{A}-\lambda \mathbf{G})^+\bold{N}w_\lambda$. Note also that the recursive computation of the eigenvectors does not involve the direct computation of $\mathbf{G}^{-1}$, which may be ill-conditioned and could lead to numerical errors. An approximation of the eigenfunction $\phi_{\lambda}$ is given by 
\begin{equation}\label{approx_eig}
\widetilde{\phi}_{\lambda}(x) = w_\lambda^\top x + \sum_{i=1}^m\left(v_{\lambda,2}\right)_ik_2(q_i,x) := w_\lambda^\top x + \tilde{h}_\lambda(x)
\end{equation} 
\noindent for all $x\in \mathcal{R}$. 
\noindent We can show that the approximation obtained in \eqref{approx_eig} verifies $\mathcal{L}\widetilde{\phi}_{\lambda} = \lambda\widetilde{\phi}_{\lambda}$ over the collocation points in $Q$.

\begin{prop}\label{prop-constraints}
Let $\widetilde{\phi}_{\lambda}$ be the approximation \eqref{approx_eig}. Then, $\mathcal{L}\widetilde{\phi}_{\lambda}(q_i) = \lambda\widetilde{\phi}_{\lambda}(q_i)$ for $i=1,\ldots,m$. 
\end{prop}
\begin{proof}
By construction, we have $\Pi \mathcal{L}\widetilde{\phi}_{\lambda} = \lambda\widetilde{\phi}_{\lambda}$. Moreover, it holds that $(\Pi\mathcal{L}\widetilde{\phi}_{\lambda})(q_i) = (\mathcal{L}\widetilde{\phi}_{\lambda})(q_i)$ since any interpolation coincides with the function itself at collocation points, which conclude the proof. 
\end{proof}
\begin{remark}
The method developed in this paper requires that $\mathbf{G}$ is invertible so that the interpolation $\mathcal{I}_{Q,k_2}$ is well-defined. If it is no longer the case, we can still define the projection

\[
\begin{split}
\Pi g & = P_{\HS_1} g + P_{\HS_{2,m}} P_{\HS_2} g \\[0.2cm]
& = \nabla^\top  g\mid_0x + \underset{h\in\HS_{2,m}} {\mathrm{argmin}}\left\|(g-\nabla^\top  g\mid_0x)-h\right\|_{\HS_2}
\end{split}
\]
for any $g\in\HS_1 \oplus \HS_2$. In this case, the matrix $\mathbf{L}_\ell$ still has a block-triangular structure, but the obtained approximated eigenfunctions do not satisfy the PDE at the collocation points (Proposition \ref{prop-constraints}). 
\end{remark}


\subsection{Connection with existing methods}

\noindent Our method bears similarity with the recent RKHS-based method developed in \cite{lee2025kernel}. In this work, the PDE $\mathcal{L}\phi_{\lambda} = \lambda\phi_{\lambda}$ is solved for any $\lambda\in\sigma(\mathbf{J}_F(0))$. More precisely, for any $\lambda\in\sigma(\mathbf{J}_F(0))$ we have 
$
\phi_\lambda(x) = w_\lambda^\top x + h_\lambda(x)
$
where $h(x)$ is the nonlinear part of the eigenfunction that verifies $\mathcal{L}h_\lambda(x) = \lambda h_\lambda(x) - w_\lambda^\top N(x)$, see e.g. \cite{deka2023path}. In \cite{lee2025kernel}, the authors compute $h$ by solving the optimization problem 
\begin{equation*}
\begin{array}{cl}
\displaystyle \min_{h\in\mathcal{H}} &\|h\|_{\mathcal{H}}\\
~~\textrm{s.t.} & \mathcal{L}h(q_i) = \lambda h(q_i) - w_\lambda^\top N(q_i),~\forall i = 1,\ldots,m,\\
\end{array}
\end{equation*}
which admits a closed form solution, which also satisfies by construction the property of Proposition \ref{prop-constraints}. In our method, $\tilde{h}_\lambda(x)$ in \eqref{approx_eig} is the solution to the  optimization problem 
\begin{equation}\label{sol_as_opt}
\begin{array}{cl}
\displaystyle \min_{h\in\mathcal{H}_{2,m}} &\|h\|_{\mathcal{H}_{2,m}}\\
~~\textrm{s.t.} & \mathcal{L}h(q_i) = \lambda h(q_i) - w_\lambda^\top N(q_i),~\forall i = 1,\ldots,m,\\
\end{array}
\end{equation}
provided that $(\mathbf{A}-\lambda \mathbf{G})$ is invertible. However, the two methods are different since $\mathcal{H}_{2,m}$ is finite-dimensional, and the solution to \eqref{sol_as_opt} cannot be obtained through the method developed in \cite{lee2025kernel}.


\section{Validation of the Lyapunov candidate through scenario-based optimization}\label{sec:validation}

Next, we will rely on a scenario-based optimization method to validate the candidate Lyapunov function obtained from the approximate eigenfunctions. In the next section, we briefly summarize the main results on nonconvex scenario-based optimization which are developed in \cite{campi2015non}.

\subsection{Scenario-based optimization}\label{sec:SP}

Let $(\Delta,\mathcal{A},\mathbb{P})$ be a probability space with probability measure $\mathbb{P}$ and consider the optimization problem
\begin{equation}
\label{original_opt}
\begin{array}{cl}
\displaystyle \max_{\theta\in\mathbb{R}^d} &f(\theta)\\
~~\textrm{s.t.} & \theta\in\Theta_{\delta},~\forall \delta\in\Delta\\
\end{array}
\end{equation}
where $f:\X\rightarrow\mathbb{R}$ is a function and $\Theta_\delta$ is a subset of $\mathbb{R}^d$ for any $\delta\in\Delta$. Note that we do not make any assumption on the problem convexity. Solving \eqref{original_opt} for any $\delta\in\Delta$ is usually challenging. The key idea is to tackle the  relaxed scenario-based problem
\begin{equation}
    \label{NCSP}
\begin{array}{cl}
\displaystyle \max_{\theta\in\mathbb{R}^d} &f(\theta)\\
~~\textrm{s.t.} & \theta\in\Theta_{\delta^{(i)}},~\forall i=1,\ldots,N,\\
\end{array}
\end{equation}
for some $N\in\mathbb{N}$, where $(\delta^{(1)},\ldots,\delta^{(N)})\in\Delta^N$ is a random variable associated with the probability measure $\mathbb{P}^N$.
The solution $\theta^*_N$ to \eqref{NCSP} is a random variable that is assumed to be unique for a fixed sample $(\delta^{(1)},\ldots,\delta^{(N)})$. Since $\theta^*_N$  depends on the sampling of scenarios, it is natural to investigate its validity with other scenarios that have not been considered. This leads to the notion of \textit{probability of violation}. 

\begin{defn}[\hspace{-0.0001cm}\cite{campi2015non}]\label{defn:viola}
The probability of violation of $\theta\in\X$ is defined as 
$
W(\theta) := \mathbb{P}\{\delta\in\Delta \mid \theta\not\in\Theta_\delta\}.
$
\end{defn}
\noindent We also recall the definition of \textit{support set}.
\begin{defn}[\hspace{-0.0001cm}\cite{campi2015non}]
A support set for \eqref{NCSP} is a subset \mbox{$I = \{\delta^{(i_1)},\ldots,\delta^{(i_k)}\} \subseteq\{\delta^{(1)},\ldots,\delta^{(N)}\}$} such that the solution to \eqref{NCSP} remains a solution for the program restricted to the constraints $\theta\in\Theta_{\delta^{(i_1)}},\ldots,\theta\in\Theta_{\delta^{(i_k)}}$.
\end{defn}

\noindent We are now in a position to state the main result providing probability guarantees that a solution to \eqref{NCSP} is a valid solution to \eqref{original_opt}.

\begin{prop}[\hspace{-0.0001cm}\cite{campi2015non}]\label{thm:SP}
Let $\beta\in (0,1)$ and $\varepsilon : \{0,\ldots,N\}\rightarrow 1$ defined by  
\begin{equation}\label{epsilon}
\varepsilon(k) := 
\left\{
\begin{array}{ll}
1&  \text{if }~ k = N,\\
1-\sqrt[N-k]{\frac{\beta}{N{N \choose k}}}&\text{otherwise}
\end{array}
\right.
\end{equation}
and assume that $I$ is a support set of \eqref{NCSP}. Then 
$$
\mathbb{P}^N\{W(\theta_N^*)\leq\varepsilon(|I|)\}>1-\beta.
$$
\end{prop}
 


\subsection{Validation of the Lyapunov function}

Let $\overline{V}$ be the candidate Lyapunov function \eqref{TrueLyap} computed with the approximated eigenfunctions \eqref{approx_eig}. Consider the sets $\mathcal{M} = \{x\in\X \mid \dot{\overline{V}}(x)<0\}$ and $\Omega_{\theta_2} \setminus \Omega_{\theta_1} = \{x\in\X \mid \theta_1<\overline{V}(x)<\theta_2\}$ for any $\theta_1,\theta_2\geq0$. In order to validate the candidate $\overline{V}$ and compute an approximation of the region of attraction, we aim at solving the optimization problem 

\begin{equation}\label{ROA_opti}
\begin{array}{cl}
\displaystyle \max_{\theta_1,\theta_2\geq 0} & \mathrm{Vol}\left(\Omega_{\theta_2} \setminus \Omega_{\theta_1}\right)\\
~~\textrm{s.t.} & \overline{\Omega_{\theta_2} \setminus \Omega_{\theta_1}} \subseteq \mathcal{M}.\\
\end{array}
\end{equation}
where $\mathrm{Vol}\left(\Omega_{\theta_2} \setminus \Omega_{\theta_1}\right) = \int_{\Omega_{\theta_2} \setminus \Omega_{\theta_1}} \mathrm{d}\mu$ refers to the volume of the set $\Omega_{\theta_2} \setminus \Omega_{\theta_1}$ and $\mu$ is the Lebesgue measure. This problem can be rewritten in the form \eqref{original_opt}
with $\theta = (\theta_1,\theta_2)\in\mathbb{R}_+^2$, $f(\theta) = \mathrm{Vol}\left(\Omega_{\theta_2} \setminus \Omega_{\theta_1}\right)$, $\Delta=\mathbb{X}$, and $\Theta_\delta = \{\theta\in\mathbb{R}_+^2 \mid \overline{V}(\delta)\leq\theta_1 \vee \overline{V}(\delta)\geq\theta_2 \vee \dot{\overline{V}}(\delta)<0\}$. 
Since it is hard to solve, we will rely on the scenario-based relaxation \eqref{NCSP} where $\Delta =\X$ is endowed with the uniform probability measure on $\X$. For the sake of simplicity and when clear from the context, we will omit the dependence on $N$ and we will denote by $\theta^* = (\theta_1^*,\theta_2^*)$ the solution to the scenario-based problem. It is clear that $\theta = (\theta_1,\theta_2)$ is a feasible solution to the obtained scenario-based problem if $\overline{V}(\delta^{(i)}) \notin [\theta_1,\theta_2]$ for all $i \in\{1,\dots,N\}$ such that $\dot{\overline{V}}(\delta^{(i)})\geq 0$ (see Figure \ref{fig:ex_alg}). Hence, an optimal solution is obtained by sorting in ascending order the values $\overline{V}(\delta^{(i)})$ for all $i \in S = \{i \, | \, \dot{\overline{V}}(\delta^{(i)})\geq0\}$, i.e. finding a map $\tau: S \to S$ such that $\overline{V}(\delta^{(\tau(i))}) \leq \overline{V}(\delta^{(\tau(i+1))})$. We have in this case the optimal solution
$$\theta^*=(\overline{V}(\delta^{(\tau(i^*))})+\xi,\overline{V}(\delta^{(\tau(i^*+1))})-\xi)$$
with
\[
i^* = \arg \max_{i \in S} \mathrm{Vol}\left(\Omega_{\overline{V}(\delta^{(\tau(i+1))})} \setminus \Omega_{\overline{V}(\delta^{(\tau(i))})}\right)
\]
and where $\xi>0$ is an arbitrarily small value which ensures that both values $\overline{V}(\delta^{(\tau(i^*))})$ and $\overline{V}(\delta^{(\tau(i^*+1))})$ do not lie in $[\theta_1^*,\theta_2^*]$. This solution is summarized in Algorithm \ref{alg:dichotomy}. For practical implementation, we shall assume that the set of scenarios is large enough and well distributed so that any volume $\mathrm{Vol}\left(\Omega_{\theta_2}\setminus\Omega_{\theta_1}\right)$ can be approximated with the number of values $\overline{V}(\delta^{(i)})$ lying in $[\theta_1,\theta_2]$. 

\begin{figure}[t]
\centering
\includegraphics[width=\linewidth]{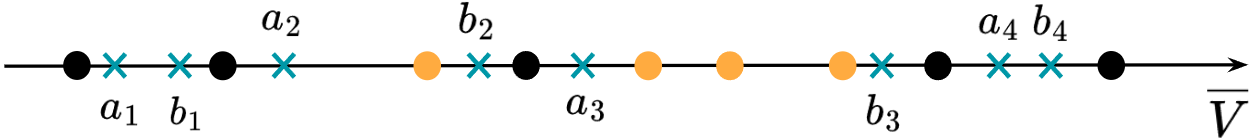}
\caption{Black dots (respectively yellow dots) correspond to scenarios $\delta^{(i)}$ such that $\dot{\overline{V}}\left(\delta^{(i)}\right)\geq0$ (respectively $\dot{\overline{V}}\left(\delta^{(i)}\right)<0$). Any couple $(a_i,b_i)\in\R_2^+$ is feasible solution to the scenario-based relaxation of \eqref{ROA_opti}}
\label{fig:ex_alg}
\end{figure}

\begin{algorithm}[b]
\begin{algorithmic}[1]
\REQUIRE $\{\delta^{(i)}\}_{i=1}^N$ and $\xi>0$ arbitrarily small.  
\ENSURE $(\theta_1^*,\theta_2^*)$ solution of the scenario relaxation of \eqref{ROA_opti}
\STATE Identify the set $S$ and compute the sorting map $\tau$
\STATE $l = 0$
\FOR{$i = 1,\ldots,N-1$}
\STATE $s = \mathrm{Vol}\left(\Omega_{\overline{V}(\delta^{(\tau(i+1))})} \setminus \Omega_{\overline{V}(\delta^{(\tau(i))})}\right)$
\IF{$ s\geq l$}
\STATE $\theta_1^* = \overline{V}(\delta^{(\tau(i))})+\xi$
\STATE $\theta_2^* = \overline{V}(\delta^{(\tau(i+1))})-\xi$
\STATE $l = s$
\ENDIF
\ENDFOR
\end{algorithmic}
\caption{Solution of the scenario relaxation of \eqref{ROA_opti}}
\label{alg:dichotomy}
\end{algorithm}


Finally, it is easy to see that a support set is given by $I=\{\delta^{(\tau(i^*))},\delta^{(\tau(i^*+1))}\}$ so that Theorem \ref{thm:SP} can be applied with
$$\varepsilon(2)=1-\sqrt[N-2]{\frac{2\beta}{N^2(N-1)}}.$$
Note that this result is independent of the dimension of the dynamical system. Moreover, the dimension only affects the computation of $\overline{V}$ and $\dot{\overline{V}}$, which scales linearly with the dimension. If we set $N = 10^{4}$ (resp. $N = 5\times10^5$) and $\beta=10^{-6}$, then the fraction of points $x\in \Omega_{\theta_2^*}\setminus\Omega_{\theta_1^*}$ that violate $\dot{\overline{V}}(x)<0$ is smaller than $\varepsilon(2) = 4.1\times 10^{-3}$ (resp. $\varepsilon(2) = 1.05 \times 10^{-4}$), with probability larger than $1-\beta$. In this case, we can conclude that $\Omega_{\theta_2^*}\setminus\Omega_{\theta_1^*}$ is a valid approximation of the region of attraction, since the probability of violation is negligible (see section \ref{sec:examples}).

\section{Numerical examples}\label{sec:examples}

In this section, we use our method to compute and validate a Lyapunov function, and obtain an approximation of the region of attraction (ROA) for several dynamical systems with a stable equilibrium. For all examples, we use the kernel $k_2$ given in Example \ref{example1}, with $\eta=1$. We solve the scenario-based problem \eqref{NCSP} with $N = 10^{4}$ in Example 1 and 2 and $N = 5\times10^{5}$ in Example 3. 


\paragraph{Example 1 (Van der Pol system)} Consider the reversed Van der Pol oscillator 
\begin{equation}\label{ex1}
\left\{
\begin{array}{rcl}
\dot x_1 & = & -x_2,  \\[0.1cm]
\dot x_2 & = & -\mu(1-9x_1^2)x_2+x_1,
\end{array}
\right.
\end{equation}
where $\mu = 1$ and $x\in \X = [-1,1]^2$. This system admits a locally stable equilibrium at the origin. An approximation of the Lyapunov function is computed via the method presented in section \ref{sec:method}, with $m=100$ collocation points $\{q_i\}_{i=1}^m$ uniformly randomly distributed over $[-0.15,0.15]^2$. As shown in Figure \ref{fig:example1}, this provides a good approximation of the ROA.



\paragraph{Example 2 (Two-machine power system)} Consider the two machine power system  \cite{vannelli1985maximal} 
\begin{equation}\label{ex2}
\left\{
\begin{array}{rcl}
\dot x_1 & = & x_2,  \\[0.1cm]
\dot x_2 & = & -\dfrac1 2 x_2 - \sin\left(3x_1+\dfrac\pi 3\right)+\dfrac{\sqrt{3}}{6},
\end{array}
\right.
\end{equation}
where $x\in\X = [-1,1]^2$, which admits a locally stable equilibrium at the origin. The Lyapunov function is computed with $m=100$ collocation points uniformly randomly distributed over $[-0.08,0.08]^2$. The approximation of the ROA is shown in Figure \ref{fig:example1}. As mentioned in Remark \ref{rem:samplingQ}, sampling the collocation points in the entire ROA and not in the largest hypercube of $\mathcal{R}$ yields a larger approximation.

\begin{figure}[h]
\centering
\subfigure[Van der Pol system]{  
\includegraphics[width=0.23\textwidth]{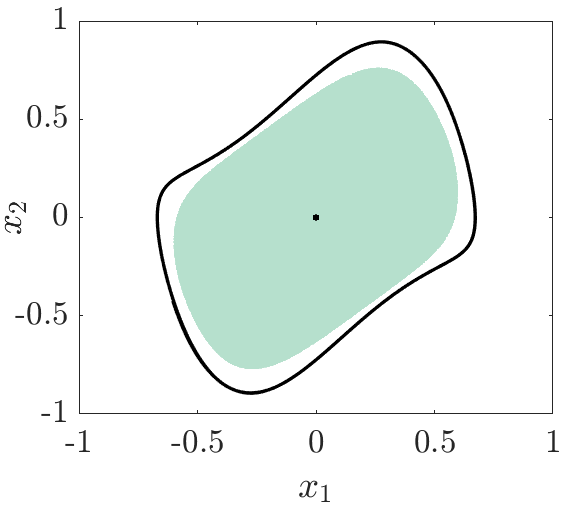}
}\hspace{-0.2cm}
\subfigure[Two-machine power system]{
\includegraphics[width=0.23\textwidth]{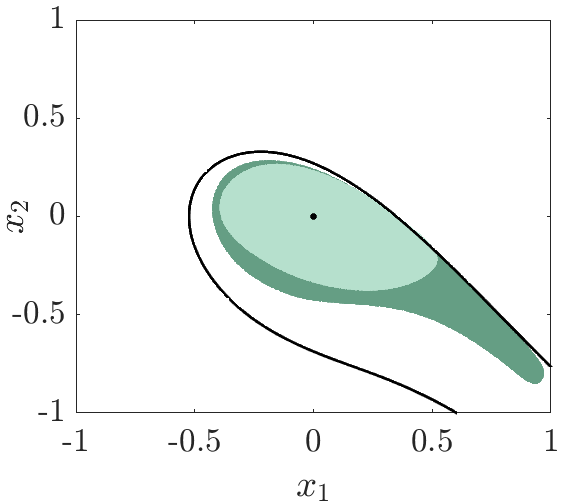}}
\caption{An approximation of the ROA for systems \eqref{ex1} and \eqref{ex2} is computed (light green area) with collocation points sampled over $[-0.15,0.15]^2$ and $[-0.08,0.08]^2$ respectively. The exact ROA is delimited by the black curve. In (b), a better approximation of the ROA is obtained when the collocation points are sampled over the hypercube $\mathcal{R})$ (dark green area).}
\label{fig:example1}
\end{figure}


\paragraph{Example 3 (Networked Van der Pol system)} We consider a network of $5$ coupled Van der Pol systems inspired by \cite{liu2025physics}. The $10$-dimensional dynamics is given by 
\begin{equation}\label{ex:VDP10}
\left\{
\begin{array}{rcl}
\dot x_{i1} &=& -x_{i2},\\[0.2cm]
\dot x_{i2} &=&x_{i1} - \alpha_i(1-x_{i1}^2)x_{i2}+\sum_{j\neq i}\beta_{ij}x_{i1}x_{j2},
\end{array}
\right.
\end{equation}
with $i=\{1,\dots,5\}$, $\X = [-4,4]^{10}$. The parameters $\alpha_i$ and $\beta_{ij}$ are randomly sampled over $[0.5,2.5]$ and $[-0.1,0.1]$, respectively. However, some parameters $\beta_{ij}$ are set to zero so that the number of nonzero entries for each $i$ is equal to $2$.
For the sake of simplicity, the system has been rescaled over $[-1,1]^{10}$. We computed a candidate Lyapunov function with $m=500$ collocation points uniformly distributed over $[-0.15,0.15]^{10}$. It is clear that the restriction of \eqref{ex:VDP10} to the $(x_{i1},x_{i2})$ cross-sections (for any $i=1,\ldots,5$) is a Van der Pol system, so that the intersection of the ROA with these cross-sections is easily computed (see the black curves in Figure \ref{fig:example3}). The approximation of the ROA is shown in Figure \ref{fig:example3}. We observe that a small region of the estimated set lies outside the true ROA. Its volume is $1.3\times10^{-6}$ while the volume of the violation set is $1.7\times10^{-6}$ (Monte-Carlo estimations). We note that the two volumes have the same order of magnitude and are bounded by $\varepsilon(2)$. We finally remark that, in some cases, the optimal solution might not be close to the origin. In such a situation, a sub-optimal solution could be preferred and the results presented in Section \ref{sec:SP} would still hold (see \cite{campi2015non}).

\begin{figure}[h]
\centering

\includegraphics[width=0.48\textwidth]{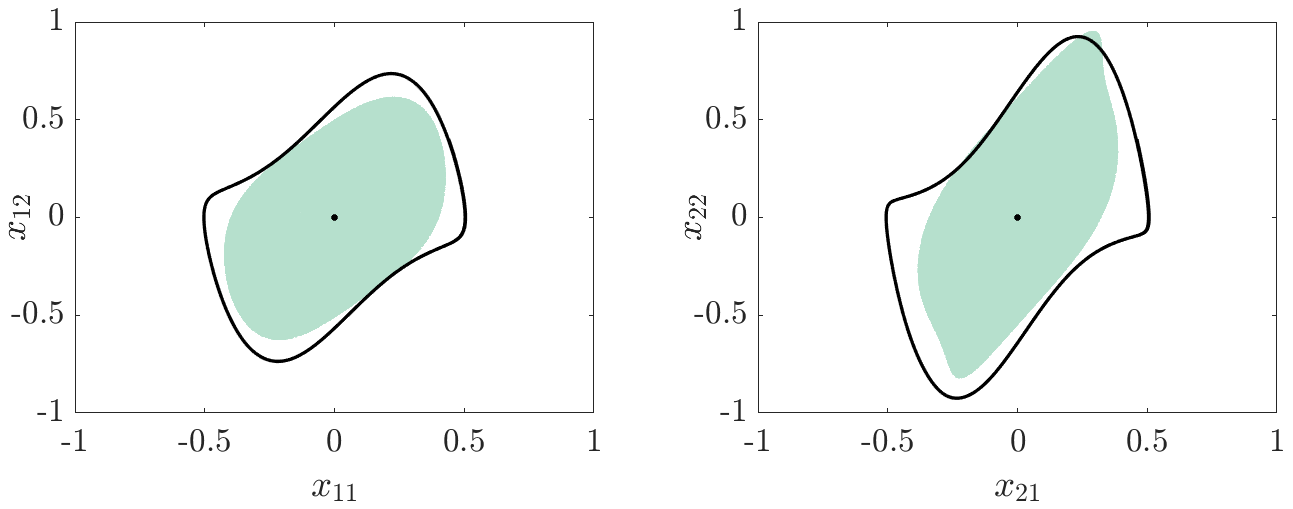}
\caption{ROA approximation of the networked Van der Pol system \eqref{ex:VDP10}. The green area is the approximation of the ROA in the $(x_{11},x_{12})$ cross-section (left) and $(x_{21},x_{22})$ cross-section (right). The black curve is the true ROA.}
\label{fig:example3}
\end{figure}

\section{Conclusions and Perspectives}\label{sec:ccl}

The Koopman operator provides a systematic method for Lyapunov function design through its eigenfunctions. In this paper, we developed a kernel-based method to compute the Koopman eigenfunctions via finite-dimensional approximation over an analytic RKHS. The analyticity property was leveraged to derive an approximation of the operator that preserves the spectrum of the linearized system while being well-suited to high-dimensional systems. Finally, we derived a scenario-based optimization technique to validate the obtained candidate Lyapunov function with probabilistic guarantees and to compute an estimation of regions of attraction. We applied our method to several systems and obtained good results for both low-dimensional and high-dimensional systems.

This work opens up several perspectives, such as optimizing over hyperparameters and collocations points, combining several Lyapunov functions, or using other validation schemes (e.g. SMT solvers \cite{gao2013dreal}). In a different setting, the scenario problem could also be recast to exploit the convergence of simulated trajectories towards the equilibrium. Finally, similarly to input-to-state Lyapunov functions, we could consider the error of approximation of the so-called lifted dynamics as an input of that system and derive rigorous stability guarantees in the same spirit as in \cite[Proposition 2.4]{mauroy2014global}.\\  

\bibliographystyle{IEEEtran}
\bibliography{ref}

@Article{mauroy2014global,
  author  = {Mauroy, A. and Mezi\'c, I.},
  title   = {Global stability analysis using the eigenfunctions of the {K}oopman operator},
  number  = {11},
  pages   = {3356--3369},
  volume  = {61},
  journal = {IEEE Tran Autom Control},
  year    = {2016},
}

@Book{mauroy2020koopman,
  author    = {Mauroy, A. and Susuki, Y. and Mezi{\'c}, I.},
  title     = {The Koopman operator in systems and control},
  publisher = {Springer},
  year      = {2020},
}

@book{paulsen2016introduction,
  title={An introduction to the theory of reproducing kernel Hilbert spaces},
  author={Paulsen, Vern I and Raghupathi, Mrinal},
  volume={152},
  year={2016},
  publisher={Cambridge university press}
}

@book{shawe2004kernel,
  title={Kernel methods for pattern analysis},
  author={Shawe-Taylor, John and Cristianini, Nello},
  year={2004},
  publisher={Cambridge university press}
}

@inproceedings{deka2023path,
  title={{Path-integral formula for computing Koopman eigenfunctions}},
  author={Deka, Shankar A and Narayanan, Sriram SKS and Vaidya, Umesh},
  booktitle={2023 62nd IEEE Conference on Decision and Control (CDC)},
  pages={6641--6646},
  year={2023},
  organization={IEEE}
}

@article{lee2025kernel,
  title={{Kernel Methods for the Approximation of the Eigenfunctions of the Koopman Operator}},
  author={Lee, Jonghyeon and Hamzi, Boumediene and Hou, Boya and Owhadi, Houman and Santin, Gabriele and Vaidya, Umesh},
  journal={Physica D: Nonlinear Phenomena},
  pages={134662},
  year={2025},
  publisher={Elsevier}
}

@Book{KhalilNonlinearControl,
  Title = {Nonlinear systems},
  Author = {Khalil, H.K.},
  Publisher = {Prentice Hall},
  Year = {2002},
  ISBN = {0-13-067389-7}
}

@inproceedings{bierwarterror,
  title={{Error bounds on analytic Koopman-based Lyapunov functions}},
  author={Bierwart, Fran{\c{c}}ois-Gr{\'e}goire and Mauroy, Alexandre},
  booktitle={2025 European Control Conference (ECC)},
  pages={1243--1248},
  year={2025},
  organization={IEEE}
}

@inproceedings{campi2015non,
  title={Non-convex scenario optimization with application to system identification},
  author={Campi, Marco C and Garatti, Simone and Ramponi, Federico A},
  booktitle={2015 54th IEEE Conference on Decision and Control (CDC)},
  pages={4023--4028},
  year={2015},
  organization={IEEE}
}

@article{pinkus2004strictly,
  title={Strictly positive definite functions on a real inner product space},
  author={Pinkus, Allan},
  journal={Advances in Computational Mathematics},
  volume={20},
  number={4},
  pages={263--271},
  year={2004},
  publisher={Springer}
}

@article{giesl2015review,
  title={{Review on computational methods for Lyapunov functions}},
  author={Giesl, Peter and Hafstein, Sigurdur},
  journal={Discrete and Continuous Dynamical Systems-B},
  volume={20},
  number={8},
  pages={2291--2331},
  year={2015},
  publisher={Discrete and Continuous Dynamical Systems-B}
}

@article{vannelli1985maximal,
  title={{Maximal Lyapunov functions and domains of attraction for autonomous nonlinear systems}},
  author={Vannelli, Anthony and Vidyasagar, Mathukumalli},
  journal={Automatica},
  volume={21},
  number={1},
  pages={69--80},
  year={1985},
  publisher={Elsevier}
}

@article{liu2025physics,
  title={{Physics-informed neural network Lyapunov functions: PDE characterization, learning, and verification}},
  author={Liu, Jun and Meng, Yiming and Fitzsimmons, Maxwell and Zhou, Ruikun},
  journal={Automatica},
  volume={175},
  pages={112193},
  year={2025},
  publisher={Elsevier}
}

@inproceedings{gao2013dreal,
  title={d{R}eal: An {SMT} solver for nonlinear theories over the reals},
  author={Gao, Sicun and Kong, Soonho and Clarke, Edmund M},
  booktitle={Automated Deduction--CADE-24: 24th International Conference on Automated Deduction, Lake Placid, NY, USA, June 9-14, 2013. Proceedings 24},
  pages={208--214},
  year={2013},
  organization={Springer}
}

@incollection{papachristodoulou2005analysis,
  title={Analysis of non-polynomial systems using the {S}um-{O}f-{S}quares decomposition},
  author={Papachristodoulou, Antonis and Prajna, Stephen},
  booktitle={Positive polynomials in control},
  pages={23--43},
  year={2005},
  publisher={Springer}
}

@article{giesl2007meshless,
  title={Meshless collocation: Error estimates with application to dynamical systems},
  author={Giesl, Peter and Wendland, Holger},
  journal={SIAM Journal on Numerical Analysis},
  volume={45},
  number={4},
  pages={1723--1741},
  year={2007},
  publisher={SIAM}
}

@article{giesl2016approximation,
  title={Approximation of {L}yapunov functions from noisy data},
  author={Giesl, Peter and Hamzi, Boumediene and Rasmussen, Martin and Webster, Kevin},
  journal={Journal of Computational Dynamics},
  volume={7},
  number={1},
  pages={57--81},
  year={2019},
  publisher={Journal of Computational Dynamics}
}

@article{kvalheim2021existence,
  title={{Existence and uniqueness of global Koopman eigenfunctions for stable fixed points and periodic orbits}},
  author={Kvalheim, Matthew D and Revzen, Shai},
  journal={Physica D: Nonlinear Phenomena},
  volume={425},
  pages={132959},
  year={2021},
  publisher={Elsevier}
}

@article{das2020koopman,
  title={Koopman spectra in reproducing kernel Hilbert spaces},
  author={Das, Suddhasattwa and Giannakis, Dimitrios},
  journal={Applied and Computational Harmonic Analysis},
  volume={49},
  number={2},
  pages={573--607},
  year={2020},
  publisher={Elsevier}
}

@article{klus2020kernel,
  title={Kernel-based approximation of the Koopman generator and Schr{\"o}dinger operator},
  author={Klus, Stefan and N{\"u}ske, Feliks and Hamzi, Boumediene},
  journal={Entropy},
  volume={22},
  number={7},
  pages={722},
  year={2020},
  publisher={MDPI}
}

@article{williams2014kernel,
  title={A kernel-based approach to data-driven Koopman spectral analysis},
  author={Williams, Matthew O and Rowley, Clarence W and Kevrekidis, Ioannis G},
  journal={arXiv preprint arXiv:1411.2260},
  year={2014}
}

@article{yadav2025approximation,
  title={Approximation of the Koopman operator via Bernstein polynomials},
  author={Yadav, Rishikesh and Mauroy, Alexandre},
  journal={Communications in Nonlinear Science and Numerical Simulation},
  volume={147},
  pages={108819},
  year={2025},
  publisher={Elsevier}
}


\end{document}